\documentclass[a4paper, 11pt]{amsart} 
\usepackage{fullpage}               
\usepackage{boxedminipage}
\usepackage{amsfonts}
\usepackage{amsmath} 
\usepackage{amssymb}
\usepackage{graphicx}
\usepackage{amsthm}
\usepackage{subfig}
\usepackage{tikz}
\usepackage{setspace}
\usepackage[colorlinks=true,linkcolor=blue,urlcolor=blue,citecolor=red]{hyperref}
\usepackage[nameinlink]{cleveref}
\usepackage{enumerate}
\hypersetup{linktocpage=true,}
\newtheorem{theorem}{Theorem}
\newtheorem{lemma}[theorem]{Lemma}
\newtheorem{conjecture}[theorem]{Conjecture}
\newtheorem{proposition}[theorem]{Proposition}
\newtheorem{corollary}[theorem]{Corollary}
\newtheorem*{claim}{Claim}
\theoremstyle{definition}
\theoremstyle{definition}
\theoremstyle{definition}

\tikzstyle{vertex}=[circle, draw, fill=black, inner sep=0pt, minimum size=4pt]
\tikzstyle{edge}=[line width=1.5pt]

\begin{document}

\title{Classifying the globally rigid edge-transitive graphs and distance-regular graphs in the plane}

\author[Sean Dewar]{Sean Dewar*}\thanks{*Johann Radon Institute for Computational and Applied Mathematics, Austrian Academy of Sciences.\\Email: \url{sean.dewar@ricam.oeaw.ac.at}}


\begin{abstract}
	A graph is said to be globally rigid if almost all embeddings of the graph's vertices in the Euclidean plane will define a system of edge-length equations with a unique (up to isometry) solution.
	In 2007,
	Jackson, Servatius and Servatius characterised exactly which vertex-transitive graphs are globally rigid solely by their degree and maximal clique number,
	two easily computable parameters for vertex-transitive graphs.
	In this short note we will extend this characterisation to all graphs that are determined by their automorphism group.
	We do this by characterising exactly which edge-transitive graphs and distance-regular graphs are globally rigid by their minimal and maximal degrees.
\end{abstract}

\maketitle

\section{Introduction}

A \emph{realisation} of a (finite simple) graph $G=(V,E)$ is any map $p :V \rightarrow \mathbb{R}^2$,
which can be considered to be a straight-line embedding of $G$ that allows edge crossings.
Two realisations $p,q$ of the same graph $G=(V,E)$ are said to be \emph{equivalent} if $\|p(x)-p(y)\| = \|q(x)-q(y)\|$ for all edges $xy\in E$,
and they are said to be \emph{congruent} if $\|p(x)-p(y)\| = \|q(x)-q(y)\|$ for all pairs of vertices $x,y\in V$.
With this, we define a realisation $p$ of $G$ to be \emph{globally rigid} if every equivalent realisation of $G$ is congruent with $p$.
It was proven in \cite{Conn05,GortlerThurstonHealey10} that global rigidity is a generic property,
i.e., if a generic realisation $p$ of $G$ (a realisation where the coordinates of $p$ considered as a vector of $\mathbb{R}^{2|V|}$ form an algebraically independent set) is globally rigid,
then all other generic realisations of $G$ are also globally rigid.
Hence we can define a graph to be \emph{globally rigid} if it has a globally rigid generic realisation.
Global rigidity has a variety of applications,
including localisation in sensor networks \cite{sensor} and molecular conformation \cite{moleconf}.

In \cite{JSS07},
Jackson, Servatius and Servatius proved the following sufficiency condition for graphs that are \emph{vertex-transitive}, i.e., their automorphism group acts transitively on their vertex set.

\begin{theorem}[{\cite[Theorem 2.2]{JSS07}}]\label{t:jss07}
	Let $G=(V,E)$ be a connected vertex-transitive graph with degree $k$.
	Then $G$ is globally rigid if and only if either:
	\begin{enumerate}[(i)]
	    \item $k \geq 6$,
	    \item $k=5$, and either the maximal clique size is at most 4 or $|V| \leq 28$,
	    \item $k=4$, and either the maximal clique size is at most 3 or $|V| \leq 11$, or
	    \item $G$ is a complete graph.
	\end{enumerate}
\end{theorem}

In this paper we shall be interested in two other closely related but distinct classes of graphs;
\emph{edge-transitive graphs} and \emph{distance-regular graphs} (see \Cref{sec:et,sec:dt} for their respective definitions).
We shall prove the following two results.

\begin{theorem}\label{t:maine}
    Let $G$ be a connected edge-transitive graph with minimum degree $\delta$ and maximum degree $\Delta$.
    Then $G$ is globally rigid if and only if either:
    \begin{enumerate}[(i)]
	    \item $\delta \geq 4$,
	    \item $\delta = 3$ and $\Delta \geq 6$, or
	    \item $G$ is one of the graphs featured in \Cref{fig:cbp}, or $G$ is a complete graph.
	\end{enumerate}
\end{theorem}

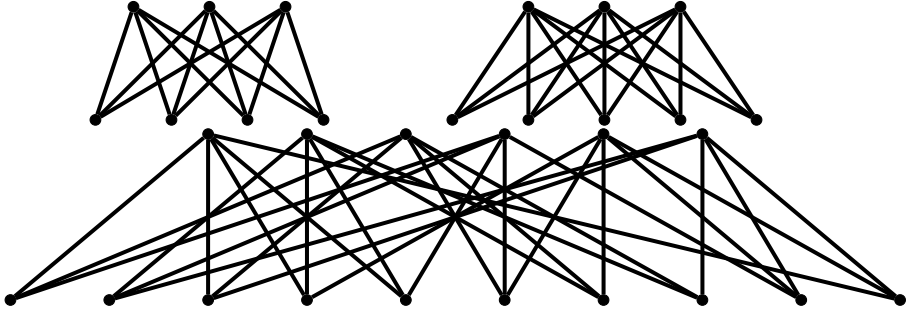
\begin{figure}[htp]
	\begin{center}
        \begin{tikzpicture}
			\node[vertex] (11) at (0.5,1) {};
			\node[vertex] (21) at (1.5,1) {};
			\node[vertex] (31) at (2.5,1) {};
			\node[vertex] (41) at (3.5,1) {};
			
			\node[vertex] (12) at (1,2.5) {};
			\node[vertex] (22) at (2,2.5) {};
			\node[vertex] (32) at (3,2.5) {};
			
			\draw[edge] (11)edge(12);
			\draw[edge] (11)edge(22);
			\draw[edge] (11)edge(32);
			
			\draw[edge] (21)edge(12);
			\draw[edge] (21)edge(22);
			\draw[edge] (21)edge(32);
			
			\draw[edge] (31)edge(12);
			\draw[edge] (31)edge(22);
			\draw[edge] (31)edge(32);
			
			\draw[edge] (41)edge(12);
			\draw[edge] (41)edge(22);
			\draw[edge] (41)edge(32);
		\end{tikzpicture}\qquad\qquad
        \begin{tikzpicture}
			\node[vertex] (01) at (0,1) {};
			\node[vertex] (11) at (1,1) {};
			\node[vertex] (21) at (2,1) {};
			\node[vertex] (31) at (3,1) {};
			\node[vertex] (41) at (4,1) {};
			
			\node[vertex] (12) at (1,2.5) {};
			\node[vertex] (22) at (2,2.5) {};
			\node[vertex] (32) at (3,2.5) {};
			
			\draw[edge] (01)edge(12);
			\draw[edge] (01)edge(22);
			\draw[edge] (01)edge(32);
			
			\draw[edge] (11)edge(12);
			\draw[edge] (11)edge(22);
			\draw[edge] (11)edge(32);
			
			\draw[edge] (21)edge(12);
			\draw[edge] (21)edge(22);
			\draw[edge] (21)edge(32);
			
			\draw[edge] (31)edge(12);
			\draw[edge] (31)edge(22);
			\draw[edge] (31)edge(32);
			
			\draw[edge] (41)edge(12);
			\draw[edge] (41)edge(22);
			\draw[edge] (41)edge(32);
		\end{tikzpicture}\qquad\qquad
		
        \begin{tikzpicture}[xscale=1.3,yscale=1.1]
			\node[vertex] (0) at (-2.5,2) {};
			\node[vertex] (1) at (-1.5,2) {};
			\node[vertex] (2) at (-0.5,2) {};
			\node[vertex] (3) at (0.5,2) {};
			\node[vertex] (4) at (1.5,2) {};
			\node[vertex] (5) at (2.5,2) {};
			
			\node[vertex] (6) at (-4.5,0) {};
			\node[vertex] (7) at (-3.5,0) {};
			\node[vertex] (8) at (-2.5,0) {};
			\node[vertex] (9) at (-1.5,0) {};
			\node[vertex] (10) at (-0.5,0) {};
			
			\node[vertex] (11) at (0.5,0) {};
			\node[vertex] (12) at (1.5,0) {};
			\node[vertex] (13) at (2.5,0) {};
			\node[vertex] (14) at (3.5,0) {};
			\node[vertex] (15) at (4.5,0) {};
			
			\draw[edge] (0)edge(6);
			\draw[edge] (0)edge(8);
			\draw[edge] (0)edge(9);
			\draw[edge] (0)edge(10);
			\draw[edge] (0)edge(15);
			
			\draw[edge] (1)edge(7);
			\draw[edge] (1)edge(9);
			\draw[edge] (1)edge(10);
			\draw[edge] (1)edge(12);
			\draw[edge] (1)edge(13);
			
			\draw[edge] (2)edge(6);
			\draw[edge] (2)edge(8);
			\draw[edge] (2)edge(11);
			\draw[edge] (2)edge(12);
			\draw[edge] (2)edge(13);
			
			\draw[edge] (3)edge(6);
			\draw[edge] (3)edge(7);
			\draw[edge] (3)edge(10);
			\draw[edge] (3)edge(11);
			\draw[edge] (3)edge(14);
			
			\draw[edge] (4)edge(9);
			\draw[edge] (4)edge(11);
			\draw[edge] (4)edge(12);
			\draw[edge] (4)edge(14);
			\draw[edge] (4)edge(15);
			
			\draw[edge] (5)edge(7);
			\draw[edge] (5)edge(8);
			\draw[edge] (5)edge(13);
			\draw[edge] (5)edge(14);
			\draw[edge] (5)edge(15);
		\end{tikzpicture}
	\end{center}
	\caption{The three special cases of \Cref{t:maine}.
	The top two graphs are the complete bipartite graphs $K_{3,4}$ and $K_{3,5}$ respectively.
	The bottom graph, which we call $H_{6,10}$, is described fully in \Cref{appendix}.
	}\label{fig:cbp}
\end{figure}

\begin{theorem}\label{t:maind}
   Let $G$ be a connected distance-regular graph with degree $k$.
   Then $G$ is globally rigid if and only if $k \geq 4$ or $G$ is complete.
\end{theorem}

By combining \Cref{t:jss07,t:maine,t:maind},
we can also characterise which graphs are globally rigid solely from their degree from the following list of graph classes:
distance-transitive graphs, arc-transitive graphs, $t$-transitive graphs, Cayley graphs (for Cayley graphs of degree 4 or 5, we will also require their maximal clique size) and strongly regular graphs.
The links between the various classes described can be seen illustrated at \cite{wiki}.
The only related graph class that is not characterised by \Cref{t:jss07,t:maine,t:maind} is the class of regular graphs.
However in this case, we can easily construct, for any choice of $k$, a $k$-regular graph that is not globally rigid  by deleting an edge from each of two $k$-regular graphs and then joining the two graphs as shown in \Cref{fig:7reg}.

\begin{figure}[htp]
	\begin{center}
        \begin{tikzpicture}[scale=1.2]
		\begin{scope}[rotate=22.5]
			\node[vertex] (1) at (1,0) {};
			\node[vertex] (2) at ({cos(45)},{sin(45)}) {};
			\node[vertex] (3) at ({cos(90)},{sin(90)}) {};
			\node[vertex] (4) at ({cos(135)},{sin(135)}) {};
			\node[vertex] (5) at ({cos(180)},{sin(180)}) {};
			\node[vertex] (6) at ({cos(225)},{sin(225)}) {};
			\node[vertex] (7) at ({cos(270)},{sin(270)}) {};
			\node[vertex] (8) at ({cos(315)},{sin(315)}) {};
			
			\draw[edge] (1)edge(2);
			\draw[edge] (1)edge(3);
			\draw[edge] (1)edge(4);
			\draw[edge] (1)edge(5);
			\draw[edge] (1)edge(6);
			\draw[edge] (1)edge(7);
			\draw[edge] (1)edge(8);
			\draw[edge] (2)edge(3);
			\draw[edge] (2)edge(4);
			\draw[edge] (2)edge(5);
			\draw[edge] (2)edge(6);
			\draw[edge] (2)edge(7);
			\draw[edge] (2)edge(8);
			\draw[edge] (3)edge(4);
			\draw[edge] (3)edge(5);
			\draw[edge] (3)edge(6);
			\draw[edge] (3)edge(7);
			\draw[edge] (3)edge(8);
			\draw[edge] (4)edge(6);
			\draw[edge] (4)edge(7);
			\draw[edge] (4)edge(8);
			\draw[edge] (5)edge(6);
			\draw[edge] (5)edge(7);
			\draw[edge] (5)edge(8);
			\draw[edge] (6)edge(7);
			\draw[edge] (6)edge(8);
			\draw[edge] (7)edge(8);
		\end{scope}
		\begin{scope}[shift={(-3,0)},rotate=22.5]
			\node[vertex] (9) at (1,0) {};
			\node[vertex] (10) at ({cos(45)},{sin(45)}) {};
			\node[vertex] (11) at ({cos(90)},{sin(90)}) {};
			\node[vertex] (12) at ({cos(135)},{sin(135)}) {};
			\node[vertex] (13) at ({cos(180)},{sin(180)}) {};
			\node[vertex] (14) at ({cos(225)},{sin(225)}) {};
			\node[vertex] (15) at ({cos(270)},{sin(270)}) {};
			\node[vertex] (16) at ({cos(315)},{sin(315)}) {};
			
			\draw[edge] (9)edge(10);
			\draw[edge] (9)edge(11);
			\draw[edge] (9)edge(12);
			\draw[edge] (9)edge(13);
			\draw[edge] (9)edge(14);
			\draw[edge] (9)edge(15);
			\draw[edge] (10)edge(11);
			\draw[edge] (10)edge(12);
			\draw[edge] (10)edge(13);
			\draw[edge] (10)edge(14);
			\draw[edge] (10)edge(15);
			\draw[edge] (10)edge(16);
			\draw[edge] (11)edge(12);
			\draw[edge] (11)edge(13);
			\draw[edge] (11)edge(14);
			\draw[edge] (11)edge(15);
			\draw[edge] (11)edge(16);
			\draw[edge] (12)edge(13);
			\draw[edge] (12)edge(14);
			\draw[edge] (12)edge(15);
			\draw[edge] (12)edge(16);
			\draw[edge] (13)edge(14);
			\draw[edge] (13)edge(15);
			\draw[edge] (13)edge(16);
			\draw[edge] (14)edge(15);
			\draw[edge] (14)edge(16);
			\draw[edge] (15)edge(16);
		\end{scope}
		
		\draw[edge] (4)edge(9);
		\draw[edge] (5)edge(16);
		\end{tikzpicture}
	\end{center}
	\caption{A 7-regular not globally rigid graph formed from two copies of $K_8$ by deleting an edge from each and joining them together}\label{fig:7reg}
\end{figure}
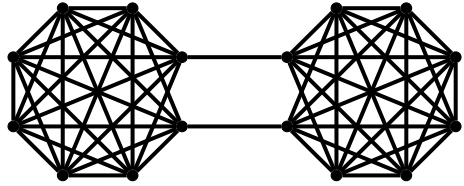

The paper is structured as follows.
In \Cref{sec:rigid} we shall briefly cover all the background knowledge of combinatorial rigidity theory that will be required throughout the paper.
In \Cref{sec:et} we shall provide some background on edge-transitive graphs before proving \Cref{t:maine}.
In \Cref{sec:dt} we shall likewise provide some background on distance-regular graphs before proving \Cref{t:maind}, and we will also achieve a quick corollary regarding strongly regular graphs from this.
In \Cref{sec:rig},
we shall classify which of the three main families of graphs we have previously mentioned are rigid (see \Cref{sec:rigid}) but not globally rigid.

\section{Combinatorial rigidity and global rigidity}\label{sec:rigid}

A realisation $p$ of a graph $G=(V,E)$ is called \emph{rigid} if there exists $\varepsilon >0$ so that any equivalent realisation $q$ where $\|q(x) - p(x)\| < \varepsilon$ for all $x \in V$ is also congruent.
This can immediately be seen to be a weaker version of global rigidity, i.e., any globally rigid realisation must also be rigid.
Like global rigidity, rigidity is a generic property \cite{AsRo78},
hence we say that a graph is \emph{rigid} if it has a generic realisation that is rigid.
It was shown by Pollaczek-Geiringer \cite{Poll1927} that rigidity can be described by a simple combinatorial characterisation.
We define a graph $G=(V,E)$ to be \emph{sparse} if $|E(X)| \leq 2|X| - 3$ for all $X \subset V$ with $|X| \geq 2$, where $E(X)$ is the set of edges with both ends in $X$,
and we say a sparse graph is \emph{tight} if $|E| = 2|V| - 3$ also.\footnote{These properties are also refereed to as \emph{(2,3)-sparse} and \emph{(2,3)-tight} in the wider literature.}

\begin{theorem}[\cite{Poll1927}]\label{t:r}
    A graph with two or more vertices is rigid if and only if it contains a spanning tight subgraph.
\end{theorem}

Later on we shall require an alternative method for characterising rigidity.
For a sparse graph $G=(V,E)$, we define a subset $X \subset V$ with $|X| \geq 2$ to be \emph{critical} if $|E(X)| = 2|X|-3$.

\begin{lemma}[see {\cite[Lemma 2.3.1]{jordan}}]\label{l:rformula}
    Let $G=(V,E)$ be a graph with $|E| \geq 1$, and let $H=(V,F)$ be a maximal sparse subgraph of $G$.
    Let $X_1,\ldots,X_t$ be the maximal critical subsets of $H$.
    \begin{enumerate}[(i)]
        \item\label{l:rformula1} A subset $X \subset V$ is a maximal critical subset of $H$ if and only if the induced subgraph $G[X] = (X,E(X))$ is a maximal rigid subgraph of $G$.
        \item\label{l:rformula2} $|F| = \sum_{i=1}^t (2|X_i|-3)$.
    \end{enumerate}
\end{lemma}

To match the combinatorial characterisation of rigid graphs,
Jackson and Jord\'{a}n \cite{JaJo05} proved an exact characterisation of globally rigid graphs by describing them as a special type of \emph{redundantly rigid} graphs, i.e., graphs $G=(V,E)$ where $G-e$ is rigid for all edges $e$.

\begin{theorem}[\cite{JaJo05}]\label{t:gr}
    A graph $G$ globally rigid if and only if it is 3-connected and redundantly rigid, or it is a complete graph.
\end{theorem}

There are various sufficient conditions for guaranteeing a graph is rigid or globally rigid.
We shall require the following three that stem from the connectivity of the graph.

\begin{theorem}[\cite{LoYe82}]\label{t:lovaszyemini}
    Every 6-connected graph is globally rigid.
\end{theorem}

A graph $G=(V,E)$ is \emph{essentially 6-connected} if:
\begin{enumerate}[(i)]
    \item $G$ is 4-connected,
    \item Given subgraphs $G_1= (V_1,E_1)$ and $G_2= (V_2,E_2)$ where $G = G_1 \cup G_2$, $|V_1\setminus V_2|\geq 3$ and  $|V_2\setminus V_1|\geq 3$,
    we have $|V_1 \cap V_2| \geq 5$.
    \item Given subgraphs $G_1= (V_1,E_1)$ and $G_2= (V_2,E_2)$ where $G = G_1 \cup G_2$, $|V_1\setminus V_2|\geq 4$ and  $|V_2\setminus V_1|\geq 4$,
    we have $|V_1 \cap V_2| \geq 6$.
\end{enumerate}
This new connectivity property gives a generalisation of \Cref{t:lovaszyemini}.

\begin{theorem}[{\cite[Theorem 3.9]{JSS07}}]\label{t:ess6}
    Every essentially 6-connected graph is globally rigid.
\end{theorem}

We say a graph $G=(V,E)$ is \emph{cyclically k-edge-connected} if for all subsets $X \subseteq V$ where the induced subgraphs $G[X]$ and $G[V \setminus X]$ both contain cycles,
there are at least $k$ edges between $X$ and $V \setminus X$.

\begin{theorem}[{\cite[Theorem 3.10]{JSS07}}]\label{t:cyc5}
   Every cyclically 5-edge-connected 4-regular graph is globally rigid.
\end{theorem}

\section{Edge-transitive graphs}\label{sec:et}

A graph $G$ is \emph{edge-transitive} if for every pair of edges $e,f$ there exists an automorphism $\pi:G \rightarrow G$ with $\pi(e) = f$.
Our interest in edge-transitive graphs can be seen to stem from the following result.

\begin{theorem}[\cite{Watkins}]\label{t:etconn}
    Let $G$ be a connected edge-transitive graph with minimum degree $k$.
    Then $G$ is $k$-connected.
\end{theorem}

\begin{lemma}\label{l:evtrans}
    Let $G=(V,E)$ be a connected edge-transitive and vertex-transitive graph with degree $k$.
    Then $G$ is globally rigid if and only if $k \geq 4$ or $G$ is complete.
\end{lemma}

\begin{proof}
    Suppose $k \leq 3$.
    Then either $|E| \leq 2|V|-3$ or $G$ is complete.
    By \Cref{t:r,t:gr},
    it follows that $G$ is globally rigid if and only if it is complete.
    
    Suppose $k \geq 4$ and $G$ is not globally rigid.
    By \Cref{t:jss07},
    we have that $k \in \{4,5\}$ and the maximal clique size of $G$ is $k$.
    As $G$ is vertex-transitive,
    it follows that every edge has end-points in a clique of size $k$.
    No automorphism of the graph can map an edge with end-points in two distinct cliques (where the end-points share no neighbours) to an edge with both end-points in same clique (where the end-points share $k-2$ neighbours),
    which contradicts that $G$ is edge-transitive.
\end{proof}

Because of \Cref{l:evtrans},
we will only need to focus on edge-transitive graphs that are not vertex-transitive.
While these graph may not be regular,
they will have strong structural properties.

\begin{proposition}[see {\cite[Proposition 15.1]{Biggs}}]\label{p:edgeNotVertex}
    Let $G = (V,E)$ be a connected edge-transitive graph that is not vertex-transitive.
    Then $G$ is a bipartite graph with parts $V_{\delta}, V_\Delta$,
    where every vertex in $V_\delta$ has degree $\delta$ and every vertex in $V_\Delta$ has degree $\Delta$.
    Furthermore, given $U \in \{ V_\delta,V_\Delta \}$, we can map any vertex of $U$ to another vertex of $U$ by an automorphism of $G$, and any automorphism of $G$ can only map vertices in $U$ to vertices in $U$.
\end{proposition}

Before beginning the proof for \Cref{t:maine},
we will need the following technical lemma.

\begin{lemma}\label{l:count}
    Let $G=(V,E)$ be a connected edge-transitive, but not vertex-transitive, graph.
    If $\delta \geq 4$, or $\delta = 3$ and $\Delta \geq 6$,
    then $|E| >2|V| - 3$.
    If $\delta \leq 3$ and $\Delta \leq 5$,
    then $|E| \leq 2|V| - 3$, or $G$ is one of the graphs featured in \Cref{fig:cbp}.
\end{lemma}

\begin{proof}
    If $\delta \leq 2$ then we must have $|E| < 2|V| - 3$;
    this is as $|V_\delta| \leq |V|-\delta$ and $|E| = \delta|V_\delta|$.
    Now suppose $\delta \geq 3$.
    As $|E| = \delta|V_\delta| = \Delta|V_\Delta|$,
    we have that $|E| = \frac{\delta}{2}|V_\delta| + \frac{\Delta}{2}|V_\Delta|$.
    If $\delta \geq 4$ then 
    \begin{align*}
        |E| \geq 2|V_\delta| + 2|V_\Delta| \geq 2|V| > 2|V|-3.
    \end{align*}
    Suppose instead that $\delta = 3$.
    We now note that
    \begin{eqnarray*}
        |E| &=& \frac{\delta}{2}|V_\delta| + \frac{\Delta}{2}|V_\Delta| \\
        &=& 2|V| - \frac{1}{2}|V_\delta| + \frac{\Delta - 4}{2}|V_\Delta| \\
        &=& 2|V| - \frac{\Delta}{6}|V_\Delta| + \frac{\Delta - 4}{2}|V_\Delta| \\
        &=& 2|V| + \frac{2\Delta - 12}{6}|V_\Delta|.
    \end{eqnarray*}
    If $\Delta \geq 6$ then $|E| \geq 2|V| > 2|V|-3$,
    and if $\Delta =3$ then $|E| \leq 2|V|-3$.
    If $\Delta = 4$ then either $|E| \leq 2|V|-3$ or $|V_\Delta| \leq 4$.
    As $3|V_\delta| = 4 |V_\Delta|$,
    the latter implies $G = K_{3,4}$.
    If $\Delta = 5$ then either $|E| \leq 2|V|-3$ or $|V_\Delta| \leq 8$.
    As $3|V_\delta| = 5 |V_\Delta|$,
    the latter implies either $|V_\delta|= 5$ and $|V_\Delta| = 3$, or $|V_\delta|= 10$ and $|V_\Delta| = 6$.
    If $|V_\delta|= 5$ and $|V_\Delta| = 3$ then $G = K_{3,5}$.
    The only edge-transitive bipartite graph with $|V_\delta|= 10$, $|V_\Delta| = 6$, $\delta = 3$ and $\Delta = 5$ is the graph $H_{6,10}$ featured in \Cref{fig:cbp} (see \cite{edgetranslist}).
\end{proof}

We will require the following concept for our next result.
A graph $G=(V,E)$ is said to be a \emph{circuit} if $|E| = 2|V|-2$ and $G-e$ is tight for every edge $e \in E$.
By \Cref{t:r}, every circuit is redundantly rigid and hence 2-connected (this follows from the observation that any separating vertex will act like a ``hinge'' for the graph and contradict rigidity).
Furthermore, it follows from \Cref{t:gr} that a circuit is globally rigid if and only if it is 3-connected.
As the circuits are the minimal dependent sets of the 2-dimensional rigidity matroid (see \cite{GrSS93} for more details), the following two properties will always hold: 
(i) if a graph $G=(V,E)$ has $|E| >2|V|-3$, then $G$ contains a circuit, and
(ii) a rigid graph is redundantly rigid if and only if every edge is contained in a circuit.

\begin{lemma}\label{l:main2}
    Let $G=(V,E)$ be a connected edge-transitive, but not vertex-transitive, graph with $\delta \geq 4$, or $\delta = 3$ and $\Delta \geq 6$.
    Then $G$ is globally rigid.
\end{lemma}

\begin{proof}
    If $\delta \geq 6$ then $G$ is globally rigid by \Cref{t:etconn,t:lovaszyemini},
    so we may assume $\delta \in \{3,4,5\}$.
    As $|E| > 2|V| - 3$ (\Cref{l:count}),
    $G$ must contain a circuit $C$.
    Choose an edge $e' \in C$ and choose for each $e \in E$ an automorphism $\phi_e :G \rightarrow G$ where $\phi_e(e') = e$.
    If we define $C_e := \phi_{e}(C)$,
    we see that every edge of $G$ lies in a circuit.
    For each edge $e \in E$,
    let $G_e$ be the maximal rigid subgraph of $G$ that contains $e$.
    As each graph $G_e$ contains every set $C_f$ for every edge $f \in E(G_e)$,
    each graph $G_e$ is redundantly rigid and contains at least four vertices.
    
    \begin{claim}
        Every graph $G_e$ is an edge-transitive bipartite graph with parts $V_{\delta,e} \subset V_\delta$ and $V_{\Delta,e} \subset V_\Delta$.
        Furthermore, for any pair $G_e,G_f$, there exists an isomorphism $\phi$ of $G$ where $\phi(G_e) = G_f$,
        $\phi(V_{\delta,e}) = V_{\delta,f}$ and $\phi(V_{\Delta,e}) = V_{\Delta,f}$.
    \end{claim}
    
    \begin{proof}
        Fix an edge $e$ and let $\phi:G \rightarrow G$ be an automorphism where $G_e$ and $\phi(G_e)$ share an edge $f$.
        The graph $G_e \cup \phi(G_e)$ is a rigid subgraph that contains an edge $f$ of $G_e$,
        since joining two rigid graphs at two points will always form a rigid graph.
        By the maximality of $G_f$ we must have $\phi(G_e) \cup G_e = G_f$,
        and thus $G_e = G_f = \phi(G_e)$.
        Now choose any two edges $e_1,e_2$ of $G_e$.
        Since $G$ is edge-transitive, there exists an automorphism $\phi':G \rightarrow G$ with $\phi'(e_1) = \phi'(e_2)$,
        hence by our previous method we have that $G_e = \phi'(G_e)$,
        i.e., each $G_e$ is edge-transitive.
        
        Choose any edges $e,f$ of $G$ and let $\phi:G \rightarrow G$ be an automorphism where $\phi(e) = f$.
        Since $\phi(G_e) \cup G_f$ is a rigid graph that contains $f$,
        we must have $\phi(G_e) \subset G_f$ by the maximality of $G_f$.
        Similarly, we have $\phi^{-1}(G_f) \subset G_e$, and hence $\phi(G_e) = G_f$.
        We now must have $\phi(V_{\delta,e}) = V_{\delta,f}$ and $\phi(V_{\Delta,e}) = V_{\Delta,f}$ by \Cref{p:edgeNotVertex}.
    \end{proof}
    
    We note that for any two edges we either have $G_e = G_f$ or $|V(G_e) \cap V(G_f)| \leq 1$.
    Hence there exists subsets $X_1,\ldots,X_t$ of $V$ and positive integers $\ell_\delta,\ell_\Delta$ so that for each $i$ we have $|X_i \cap V_\delta| = \ell_\delta$ and $|X_i \cap V_\Delta| = \ell_\Delta$,
    and the induced subgraphs $G[X_1], \ldots, G[X_t]$ are the pairwise edge-disjoint maximal rigid subgraphs of $G$.
    As every graph $G[X_i]$ is redundantly rigid, they each must have a minimal degree of at least 3.
    If a vertex $x \in V_\delta$ were to be contained in two distinct sets $X_i,X_j$,
    then $x$ must have degree at least 3 in both $G[X_i]$ and $G[X_j]$,
    which in turn implies that the degree of $x$ in $G$ would be at least 6 as $G[X_i]$ and $G[X_j]$ are edge-disjoint,
    contradicting that $\delta < 6$.
    Hence every vertex in $V_\delta$ lies in exactly one set $X_i$,
    and so $\ell_\Delta \geq \delta$ and $|V_\delta| = t\ell_\delta$.
    
    Given a vertex $x \in V_{\Delta}$,
    define $m := |\{ X_i : x \in X_i\}|$;
    the choice of vertex of $V_\Delta$ here is irrelevant, 
    since any vertex of $V_\Delta$ can be mapped to another by an automorphism of $G$.
    Since every vertex of $V_\Delta$ lies in $m$ of the sets $X_1,\ldots,X_t$,
    we note that $m|V_\Delta| = t\ell_\Delta$.
    By combining this with $|V_\delta| = t\ell_\delta$ we have that $|X_i| = \ell_\delta + \ell_\Delta = \frac{|V_\delta| + m|V_\Delta|}{t}$ for every $i$.
    
    \begin{claim}
        $G$ is rigid.
    \end{claim}
    
    \begin{proof}
        Now suppose that $G$ is not rigid.
        As $G$ is connected and each vertex in $V_\delta$ lies in exactly one of the sets $X_1, \ldots, X_t$,
        $m = 1$ implies $t = 1$ and $G$ is rigid,
        hence $m \geq 2$.
        We now note that
        \begin{eqnarray*}
            \sum_{i=1}^t 2|X_i| - 3 &=& \sum_{i=1}^t \sum_{v \in X_i \cap V_\delta} \left( 2 - \frac{3}{|X_i| } \right) + \sum_{i=1}^t \sum_{v \in X_i \cap V_\Delta}\left(  2 - \frac{3}{|X_i| }\right) \\
            &=& \sum_{v \in V_\delta}  \left(2 - \frac{3 t}{|V_\delta| + m|V_\Delta| }\right) + m \sum_{v \in V_\Delta} \left( 2 - \frac{3 t}{|V_\delta| + m|V_\Delta| }\right) \\
            &=& \left(|V_\delta| + m|V_\Delta| \right)\left( 2 - \frac{3 t}{|V_\delta| + m|V_\Delta| } \right) \\
            &=& 2|V| + 2(m-1)|V_\Delta| - 3t \\
            &=& 2|V| + 2(m-1)|V_\Delta| - 3\frac{m|V_\Delta|}{\ell_\Delta}\\
            &\geq& 2|V| + |V_\Delta| \left( 2(m-1) - 3m/\delta \right)  \qquad \qquad \text{(as $\ell_\Delta \geq \delta$)}\\
            &\geq& 2|V|  \qquad \qquad \text{(as $m \geq 2$ and $\delta \geq 3$)}.
        \end{eqnarray*}
        Let $H = (V,F)$ be a maximal sparse subgraph of $G$.
        By \Cref{l:rformula}(\ref{l:rformula1}),
        the maximal critical sets of $H$ are exactly $X_1,\ldots,X_t$.
        However it now follows from \Cref{l:rformula}(\ref{l:rformula2}) that $|F|>2|V|-3$, a contradiction.
    \end{proof}
    
    Since every edge of $G$ lies in a circuit,
    we have that $G$ is redundantly rigid.
    Hence by \Cref{t:gr,t:etconn},
    $G$ is globally rigid.
\end{proof}

\begin{proof}[Proof of \Cref{t:maine}]
    First suppose that $\delta \leq 3$, with $\Delta \leq 5$ if $\delta = 3$,
    and further suppose that $G$ is not complete or isomorphic to any of the graphs in \Cref{fig:cbp}.
    By \Cref{l:count},
    $|E| \leq 2|V|-3$,
    thus $G$ is not globally rigid by \Cref{t:r,t:gr}.
    If $\delta \geq 4$ or $\delta=3$ and $\Delta \geq 6$,
    then $G$ is globally rigid by \Cref{l:evtrans,l:main2}.
    
    Now suppose $G$ is one of the graphs given in \Cref{fig:cbp}.
    Then $G$ contains a spanning tight subgraph (and hence is rigid by \Cref{t:r});
    see \cite{pebble} for a polynomial-time algorithm that can be used to construct a spanning tight subgraph of $G$.
    Since $|E| > 2|V| -3$ and $G$ is edge-transitive,
    every edge of $G$ is contained in a circuit.
    Hence, $G$ is redundantly rigid.
    By \Cref{t:etconn},
    $G$ is 3-connected,
    thus $G$ is globally rigid by \Cref{t:gr}.
\end{proof}

\section{Distance-regular graphs}\label{sec:dt}

Let $G = (V,E)$ be a graph and $d(x,y)$ be the distance between two vertices of $G$.
For any non-negative integers $i,j$ and any (possibly not distinct) vertices $x,y$,
define the set
\begin{align*}
    D_{i,j}(x,y) := \{ z \in V : d(x,z) = i, d(y,z) = j \}.
\end{align*}
The graph $G$ is said to be \emph{distance-regular} if for any triple $i,j,k$ of non-negative integers,
there exists a non-negative integer $a_{i,j,k}$ such that for any vertices $x,y$ where $d(x,y) = k$,
we have $|D_{i,j}(x,y)| = a_{i,j,k}$.
It is immediate that every distance-regular graph is regular. 
Distance-regular graphs also have very nice connectivity properties.

\begin{theorem}[\cite{BK09}]\label{t:distconn}
   Any distance-regular graph $G$ with degree $k$ is $k$-connected.
   Furthermore, any separating set of $G$ of size $k$ is the neighbourhood of a vertex.
\end{theorem}

By \Cref{t:lovaszyemini,t:distconn},
every distance-regular graph with degree at least 6 is globally rigid.
The Bannai-Ito conjecture \cite{bannaiitoconj} -- i.e., for any $k \geq 3$, there are only finitely many distance-regular graphs with degree $k$ --  was proven recently by Bang, Dubickas, Koolen, Moulton \cite{bannaiitoproof},
hence there are only finitely many graphs that need to be checked to prove \Cref{t:maind}.
However, the full list of all distance-regular graphs with degree 4 and higher is still unknown, so checking the finite remaining special cases is unfortunately still out of reach.
For more information about the ongoing effort to determine all the distance-regular graphs with degree 4,
see \cite{BK99}.
Instead we shall turn to essential and cyclic connectivity to prove \Cref{t:maind}.

\begin{proof}[Proof \Cref{t:maind}]
    First let $k \geq 5$.
    If $k \geq 6$ then $G$ is globally rigid by \Cref{t:lovaszyemini,t:distconn}.
    Suppose instead that $k=5$.
    Choose any two subgraphs $G_1= (V_1,E_1)$ and $G_2= (V_2,E_2)$ where $G = G_1 \cup G_2$, $|V_1\setminus V_2|\geq t$ and $|V_2\setminus V_1|\geq t$ for some positive integer $t$.
    By \Cref{t:distconn},
    we have $|V_1 \cap V_2| \geq 5$ if $t =1$ (i.e., $G$ is 5-connected) and $|V_1 \cap V_2| \geq 6$ if $t \geq 2$.
    Hence $G$ is essentially 6-connected, and thus globally rigid by \Cref{t:ess6}.
    
    Now let $k = 4$.
    If $|V| \leq 7$,
    then $G$ is either the complete graph on 5 vertices or the octohedron graph;
    see \cite{BK99} for more details.
    It can be quickly checked that both graphs are redundantly rigid and 3-connected,
    and hence $G$ globally rigid by \Cref{t:gr}.
    Suppose $|V| \geq 8$,
    and choose $X \subset V$ such that $G[X]$ and $G[V\setminus X]$ both contain a cycle.
    Let $t$ be the number of edges between $X$ and $V \setminus X$.
    Since $|V| \geq 8$ and $G$ is 4-regular,
    either one of $X$ or $V \setminus X$ is a clique of size 3 and $t = 6$,
    or $|X| , |V \setminus X| \geq 4$.
    Suppose the latter holds and $t <5$.
    By \Cref{t:distconn},
    $t = 4$ and $G$ is 4-connected.
    The edges separating $X$ and $V \setminus X$ must be independent (or else $G$ would be 3-connected),
    hence there exists separating set $S = \{x_1,x_2,y_1,y_2\}$ of $G$ so that $x_1,x_2 \in X$, $y_1,y_2 \in V \setminus X$, and every vertex in $V \setminus S$ can be neighbours with at most three of the vertices in $S$.
    By \Cref{t:distconn} we must have that $|S| \geq 5$, a contradiction,
    hence $t = 5$ and $G$ is cyclically 5-edge-connected.
    Thus $G$ is globally rigid by \Cref{t:cyc5}.
\end{proof}

\section{Classifying the rigid but not globally rigid graphs}\label{sec:rig}

In this section, we will show that there are only two graphs that are edge-transitive and/or distance-regular which are rigid but not globally rigid.
These graphs are exactly the two cubic graphs featured in \Cref{fig:rigidcubic}.

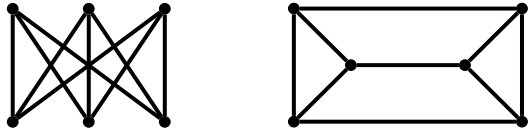
\begin{figure}[htp]
	\begin{center}
        \begin{tikzpicture}[scale=1]
			\node[vertex] (11) at (1,1) {};
			\node[vertex] (21) at (2,1) {};
			\node[vertex] (31) at (3,1) {};
			
			\node[vertex] (12) at (1,2.5) {};
			\node[vertex] (22) at (2,2.5) {};
			\node[vertex] (32) at (3,2.5) {};
			
			\draw[edge] (11)edge(12);
			\draw[edge] (11)edge(22);
			\draw[edge] (11)edge(32);
			
			\draw[edge] (21)edge(12);
			\draw[edge] (21)edge(22);
			\draw[edge] (21)edge(32);
			
			\draw[edge] (31)edge(12);
			\draw[edge] (31)edge(22);
			\draw[edge] (31)edge(32);
		\end{tikzpicture}\qquad\qquad
        \begin{tikzpicture}[scale=0.75]
			\node[vertex] (11) at (-1,0) {};
			\node[vertex] (21) at (-2,1) {};
			\node[vertex] (31) at (-2,-1) {};
			
			\node[vertex] (12) at (1,0) {};
			\node[vertex] (22) at (2,1) {};
			\node[vertex] (32) at (2,-1) {};
			
			\draw[edge] (11)edge(12);
			\draw[edge] (21)edge(22);
			\draw[edge] (31)edge(32);
			
			\draw[edge] (11)edge(21);
			\draw[edge] (21)edge(31);
			\draw[edge] (31)edge(11);
			
			\draw[edge] (12)edge(22);
			\draw[edge] (22)edge(32);
			\draw[edge] (32)edge(12);
		\end{tikzpicture}
	\end{center}
	\caption{The only two cubic graphs that are rigid but not complete;
	$K_{3,3}$ (left) and $K_2 \times K_3$ (right).
	}\label{fig:rigidcubic}
\end{figure}

In fact, the graphs featured in \Cref{fig:rigidcubic} are the only rigid but not globally rigid cubic graphs.

\begin{lemma}\label{l:cubic}
    The only globally rigid cubic graph is $K_4$,
    and the only rigid but not globally rigid cubic graphs are the two pictured in \Cref{fig:rigidcubic}.
\end{lemma}

\begin{proof}
    Let $G=(V,E)$ be a cubic rigid graph that is not complete.
    By the hand-shaking lemma, we have that $|E| = \frac{3}{2}|V|$.
    If $|E| < 2|V| -3$ then $G$ is flexible (\Cref{t:r}),
    and if $|V|=4$ then $G$ is complete,
    hence $|V|=6$.
    Thus $G$ is either $K_{3,3}$ or $K_2 \times K_3$ (see \Cref{fig:rigidcubic}).
    Both $K_{3,3}$ and $K_2 \times K_3$ are tight,
    and hence are rigid but not globally rigid by \Cref{t:r,t:gr}.
\end{proof}

The class of vertex-transitive graphs that are rigid but not globally rigid was characterised by Jackson, Servatius and Servatius.

\begin{corollary}[{\cite[Corollary 2.5]{JSS07}}]\label{c:jss07}
	Let $G=(V,E)$ be a connected vertex-transitive graph of degree $k$.
	Suppose $G$ is not globally rigid.
	Then $G$ is rigid if and only if either:
	\begin{enumerate}[(i)]
	    \item $k=5$, $G$ contains a clique of size 5 and $30 \leq |V| \leq 38$,
	    \item $k=4$, $G$ contains a clique of size 4 and $12 \leq |V| \leq 15$, or
	    \item $G$ is one of the graphs featured in \Cref{fig:rigidcubic}.
	\end{enumerate}
\end{corollary}

We will now give the analogues of this result for both edge-transitive graphs and distance-regular graphs.

\begin{corollary}\label{c:maine}
    Let $G=(V,E)$ be a connected edge-transitive graph.
    Suppose $G$ is not globally rigid.
    Then $G$ is rigid if and only if $G$ is one of the graphs featured in \Cref{fig:rigidcubic}.
\end{corollary}

\begin{proof}
    Let $\delta$ and $\Delta$ be the minimum and maximum degree of $G$ respectively.
    By \Cref{t:maine},
    either $\delta \leq 2$, or $\delta \leq 3$ and $\Delta \leq 5$.
    If $G$ is vertex-transitive then it is one of the graphs featured in \Cref{fig:rigidcubic} by \Cref{c:jss07},
    so we will assume $G$ is not vertex-transitive.
    By \Cref{p:edgeNotVertex},
    $G$ is a bipartite graph with parts $V_\delta$ and $V_\Delta$.
    First suppose $\delta \leq 2$.
    The only rigid graphs with 3 or less vertices are the complete graphs,
    hence $|V| \geq 4$.
    Since $|E| = \delta|V_\delta| \leq \delta(|V| - \delta) < 2|V|-3$,
    $G$ cannot be rigid by \Cref{t:r}.
    Now suppose $\delta = 3$ and $\Delta \leq 5$.
    As shown in \Cref{l:count},
    we have $|E| = 2|V| + \frac{2\Delta - 12}{6}|V_\Delta|$.
    Since $G$ is not any of the graphs in \Cref{fig:cbp} (which are the only edge-transitive graphs with their corresponding values for $\delta,\Delta, |V_\delta|,|V_\Delta|$; see \cite{edgetranslist}),
    we have that $|E| \leq 2|V| - 3$, with equality if and only if either $\Delta = 5$, $|V_\delta| = 15$ and $|V_\Delta|=9$, or $\Delta = 3$, $|V_\delta| = 3$ and $|V_\Delta| = 3$.
    There are no edge-transitive graphs with the first set of parameters (see \cite{edgetranslist}),
    and the only graph satisfying the second set is $K_{3,3}$.
\end{proof}

\begin{corollary}\label{c:maind}
   Let $G=(V,E)$ be a connected distance-regular graph.
   Suppose $G$ is not globally rigid.
   Then $G$ is rigid if and only if $G$ is one of the graphs featured in \Cref{fig:rigidcubic}.
\end{corollary}

\begin{proof}
    Let $k$ be the degree of $G$.
    As $G$ is not globally rigid,
    we must have that $k \leq 3$ by \Cref{t:maind}.
    The result holds for $k=3$ by \Cref{l:cubic},
    hence we need only check $k \leq 2$.
    The only $k$-regular connected graphs for $k \leq 2$ that are not complete graphs are the cycles,
    and no cycle with more than three vertices is rigid by \Cref{t:r}.
\end{proof}

\section{A conjecture for higher dimensional global rigidity}

Rigidity and global rigidity can also be considered in higher dimensional spaces;
see \cite{JaJo05} for more details.
With this in mind, the author would conjecture the following higher-dimensional analogue of \Cref{t:jss07,t:maine,t:maind}.

\begin{conjecture}\label{conj}
	There exist functions $f_1,f_2,f_3 : \mathbb{N} \rightarrow \mathbb{N}$ and $g : \mathbb{N} \times \mathbb{N} \rightarrow \mathbb{N}$ such that the following holds:
	\begin{enumerate}[(i)]
		\item Any vertex-transitive graph with degree at least $f_1(d)$ is globally rigid in $\mathbb{R}^d$.
		\item Any edge-transitive graph with minimum degree $\delta \geq f_2(d)$ and maximum degree $\Delta \geq g(\delta, d)$ is globally rigid in $\mathbb{R}^d$.
		\item Any distance-regular graph with degree at least $f_3(d)$ is globally rigid in $\mathbb{R}^d$.
	\end{enumerate}	
\end{conjecture}

It is known that a complete bipartite graph $K_{m,n}$ is globally rigid in $\mathbb{R}^d$ if and only if $m,n \geq d+1$ and $m +n \geq \binom{d+1}{2}$; see \cite{CGT21} for an explicit proof of this.
Furthermore, $K_{m,n}$ is edge-transitive and, if $m=n$, also vertex-transitive and distance-regular.
It hence follows that if the functions $f_1,f_2,f_3,g$ do exist, they satisfy the inequalities $f_1(d),f_2(d),f_3(d) \geq  d+1$ and $g(f_2(d),d) \geq \binom{d+1}{2} - f_2(d)$.

\subsection*{Acknowledgements}
The author was supported by the Austrian Science Fund (FWF): P31888.

\appendix
\section{The graph \texorpdfstring{$H_{6,10}$}{H(6,10)}}\label{appendix}

We first label the vertices of $H_{6,10}$ as the integers $\{0,\ldots,15\}$,
with $\{0,1,2,3,4,5\}$ and $\{6,7,8,9,10,11,12,13,14,15\}$ forming the parts of the bipartite graph.
With respect to this vertex labelling,
the adjacency list of $H_{6,10}$ is
\begin{gather*}
    0: \{6, 8, 9, 10, 15\}, \qquad 1: \{7, 9, 10, 12, 13\}, \qquad  2: \{6, 8, 11, 12, 13\},\\
    3: \{6, 7, 10, 11, 14\},\qquad 4: \{9, 11, 12, 14, 15\}, \qquad 5: \{7, 8, 13, 14, 15\}, \\ 
    6: \{0, 2, 3\}, \qquad 7: \{1, 3, 5\},\qquad  8: \{0, 2, 5\},\qquad  9: \{0, 1, 4\},\qquad  10: \{0, 1, 3\},\\ 11: \{2, 3, 4\}, \qquad 12: \{1, 2, 4\},\qquad  13: \{1, 2, 5\},\qquad  14: \{3, 4, 5\}, \qquad 15: \{0, 4, 5\}.
\end{gather*}
It has an automorphism group of size 60 and a diameter of 3.
It also has the property that any two distinct vertices of degree 5 share exactly two neighbours,
and any two distinct vertices of degree 3 share either one or two neighbours.
\end{document}